\documentclass{amsart}
\usepackage{amssymb, url}

\newtheorem{theorem}{Theorem}[section]
\newtheorem{lemma}[theorem]{Lemma}

\theoremstyle{definition}

\newtheorem{proposition}[theorem]{Proposition}
\newtheorem{corollary}[theorem]{Corollary}

\theoremstyle{remark}
\newtheorem{remark}[theorem]{Remark}

\numberwithin{equation}{section}

\begin{document}
\title[Local holomorphic maps preserving $(p,p)$-forms]
{On local holomorphic maps between Hermitian manifolds preserving $(p,p)$-forms}

\author[S. T. Chan]{Shan Tai Chan}
\address{State Key Laboratory of Mathematical Sciences \& Institute of Mathematics, Academy of Mathematics and Systems Science, Chinese Academy of Sciences, Beijing 100190, China.}
\email{mastchan@amss.ac.cn}
\thanks{The author is partially supported by a program of the Chinese Academy of Sciences}


\begin{abstract}
In this article, we generalize some results in Chan-Yuan [Ann.\,Sc. Norm.\,Super.\,Pisa\,Cl.\,Sci.\,(5) 26 (2025), 619--644] to local holomorphic maps between Hermitian manifolds preserving $(p,p)$-forms. In particular, we obtain further rigidity theorems and non-existence theorems for such maps.
\end{abstract}

\maketitle

\section{Introduction}
In Chan-Yuan \cite{CY25} we have studied local holomorphic maps between K\"ahler manifolds preserving $(p,p)$-forms, namely, the $p$-th exterior power of their K\"ahler forms. In particular, we have established several rigidity theorems and non-existence theorems.
Later on, we realize that by similar methods, some of these results can be generalized to the case of local holomorphic maps between Hermitian manifolds preserving $(p,p)$-forms, namely the $p$-th exterior power of their \text{Hermitian} forms.
In this article, we aim at providing further rigidity theorems and \text{non-existence} \text{theorems} for local holomorphic maps between Hermitian manifolds preserving $(p,p)$-forms.

We will follow the settings in Chan-Yuan \cite{CY25} throughout this article. Recall that a (finite-dimensional) complex space form is a connected complete K\"ahler manifold of constant holomorphic sectional curvature. We write $X=X(c)$ to indicate that $X$ is of constant holomorphic sectional curvature $c$. The universal covering space of a complex space form of complex dimension $n$ is either the complex Euclidean space $\mathbb C^n$, the complex unit ball $\mathbb B^n$ or the complex projective space $\mathbb P^n$. We say that two complex space forms $M$ and $N$ are of the same type if $M=M(c)$ and $N=N(c')$ such that either $c$ and $c'$ are of the same sign or $c=c'=0$. We say that $M$ and $N$ are of different types if $M$ and $N$ are not of the same type.

In \cite{Ca53}, Calabi has proved that there does not exist a local holomorphic isometric embedding between complex space forms of different types.
This motivates our study in Chan-Yuan \cite{CY25}, and we have partially generalized Calabi's result to the case of local holomorphic maps preserving $(p,p)$-forms (cf.\,Chan-Yuan \cite[Theorem 1.4]{CY25}).
Recently, Arezzo-Li-Loi \cite[Corollary 1.7]{ALL25} have proved that there does not exist a Ricci-flat non-compact K\"ahler submanifold of any finite dim. complex projective space equipped with the Fubini–Study metric. Hence, it follows from Chan-Yuan \cite[Theorem 1.4 and Remark 3.9]{CY25} and Arezzo-Li-Loi \cite[Corollary 1.7]{ALL25} that if $(M,\omega_M)$ and $(N,\omega_N)$ are complex space forms of complex dimensions $m$ and $n$ respectively such that $M$ and $N$ are of different types, then for any real constant $\lambda>0$ and any integer $p$, $1\le p\le \min\{m,n\}$, there does not exist a local holomorphic map $F:(M;x_0) \to (N;F(x_0))$ such that $F^*\omega_N^p = \lambda \omega_M^p$ holds. This yields a complete generalization of Calabi's result.
(Noting that in Chan-Yuan \cite[Theorem 1.4]{CY25} we exclude the case where the universal cover of $M$ is biholomorphic to $\mathbb{C}^m$, $N= \mathbb{P}^n$ and $2\le p=m<n$.)

\section{Preliminaries}
Given a Hermitian manifold $M$ of complex dimension $m$, we write $\omega_M$ for its Hermitian form, and $g_M$ for the corresponding Hermitian metric.
Fix $x_0\in M$ and let $w=(w_1,\ldots,w_m)$ be local holomorphic coordinates around $x_0$. For $1\le p\le m$, we write ${\bf A}=\{I=(i_1,\ldots,i_p): 1\le i_1<\cdots<i_p \le m\}$ and $\omega_M^p =  (\sqrt{-1})^p\sum_{I,J\in {\bf A}}(\omega_M^p)_{I\overline J} dw^I \wedge d\overline{w^J}$ in terms of the local holomorphic coordinates.
In what follows, for any $v\in \wedge^p T_{x_0}(M)$, we write $v=\sum_{I\in {\bf A}} v_{I} {\partial \over \partial w_{i_1}}\wedge\cdots \wedge {\partial \over \partial w_{i_p}}$, $I=(i_1,\ldots,i_p)$, and write
$\omega_M^p(x_0)(v,\overline v)
:= \sum_{I,J\in {\bf A}} (\omega_M^p)_{I\overline J}(x_0) v_{I} \overline{v_{J}}$.

For any complex manifold $X$, we write $T_X$ for the holomorphic tangent bundle of $X$, $\Theta_E$ for the curvature operator of the Chern connection on a Hermitian \text{holomorphic} vector bundle $(E,h)$ over $X$, and $\Lambda(X)$ for the Umehara algebra on $X$ as defined in Umehara \cite{Um88}.
More precisely, $\Lambda(X)$ is the associative algebra of real analytic functions on $X$ that consists of real linear combinations of functions $f\overline{g}+g\overline{f}$ for holomorphic functions $f$ and $g$ on $X$.

For any irreducible bounded symmetric domain $D$, we denote by $g_{D}$ the canonical K\"ahler-Einstein metric on $D$ so that the minimum of the holomorphic sectional curvatures of $(D,g_{D})$ equals $-2$. Given any bounded symmetric domain $\Omega$, we say that $g'_\Omega$ is a \emph{canonical K\"ahler metric} on $\Omega$ whenever $g'_\Omega=\sum_{j=1}^k \lambda_j {\rm Pr}_j^* g_{\Omega_j}$ for some real constants $\lambda_j>0$, $1\le j\le k$, where $\Omega_j$, $1\le j\le k$, are the irreducible factors of $\Omega$, and ${\rm Pr}_j:\Omega \to \Omega_j$ is the canonical projection onto the $j$-th irreducible factor $\Omega_j$, $1\le j\le k$.
We denote by $g_{\mathbb P^n}$ the Fubini-Study metric on the complex projective space $\mathbb P^n$ of complex dimension $n$ such that $(\mathbb P^n,g_{\mathbb P^n})$ is of constant holomorphic sectional curvature $2$.

\section{Rigidity theorem for local holomorphic maps between Hermitian manifolds preserving $(p,p)$-forms}\label{Sec:Local holo map pp form}
Let $(X,g)$ be a complex $n$-dimensional Hermitian manifold with the Hermitian form $\omega$, namely, $\omega=g(J\cdot,\cdot)$ with $J$ being the induced almost complex structure on $(X,g)$.
For simplicity, in the present article we write $(X,\omega)$ to indicate that $\omega=:\omega_X$ is the Hermitian form.
In terms of local holomorphic coordinates $(z_1,\ldots,z_n)$ on $X$, we may write $\omega=\sqrt{-1}\sum_{i,j=1}^n g_{i\overline j} dz_i\wedge d\overline{z_j}$ locally.
Then, for $1\le p\le n$ we have
\[ \omega^p = (\sqrt{-1})^p p! 
\sum_{\begin{tiny}\begin{split}
&1\le i_1<\cdots<i_p\le n,\\
&1\le j_1<\cdots<j_p\le n
\end{split}\end{tiny}}
 \det\begin{pmatrix} g_{i_s\overline{j_t}} \end{pmatrix}_{1\le s,t\le p}
 (dz_{i_1}\wedge d\overline{z_{j_1}}) \wedge \cdots \wedge 
 (dz_{i_p}\wedge d\overline{z_{j_p}}). \]

Let $(M,g_M)$ and $(N,g_N)$ be Hermitian manifolds of complex dimensions $m$ and $n$ respectively.
Fix some point $x_0\in M$.
Let $F:(M;x_0) \to (N;F(x_0))$ be a germ of holomorphic map such that
\begin{equation}\label{Eq:General_pp}
F^*\omega_N^p = \lambda \omega_M^p,
\end{equation}
for some integer $p$, $1\le p\le \min\{m,n\}$, and some real constant $\lambda>0$.
Then, such a germ of holomorphic map $F$ is of rank at least $p$ near $x_0$.

Write $w=(w_1,\ldots,w_m)$ (resp.\,$z=(z_1,\ldots,z_n)$) for the local holomorphic coordinates on $M$ (resp.\,$N$) around the point $x_0$ (resp.\,$F(x_0)$). We write
\[ \omega_M = \sqrt{-1}\sum_{i,j=1}^m (g_M)_{i\overline j} dw_i\wedge d\overline{w_j},\quad \omega_N = \sqrt{-1} \sum_{i,j=1}^n (g_N)_{i\overline j} dz_i\wedge d\overline{z_j}, \]
$F(w) = (F_1(w),\ldots,F_n(w))$ locally around $x_0$ in terms of the local holomorphic coordinates, and
\[ {\partial(F_{i_1},\ldots,F_{i_p})\over \partial(w_{l_1},\ldots,w_{l_p})}
:= \begin{pmatrix}
{\partial F_{i_1}\over \partial w_{l_1}} & \cdots & {\partial F_{i_1}\over \partial w_{l_p}}\\
\vdots & \ddots & \vdots\\
{\partial F_{i_p}\over \partial w_{l_1}} & \cdots & {\partial F_{i_p}\over \partial w_{l_p}}
\end{pmatrix}
 \]
for $1\le i_1<\cdots<i_p\le n$ and $1\le l_1<\cdots <l_p\le m$.
Then, we have
\[   F^*\omega_N^p
= C_p\cdot  
\sum_{\begin{tiny}\begin{split}
&1\le l_1<\cdots<l_p\le m,\\
&1\le k_1<\cdots<k_p\le m \end{split}\end{tiny}}
\Theta_{l_1,\ldots,l_p;\overline{k_1},\ldots,\overline{k_p}}(w)
dw_{l_1}\wedge \cdots \wedge dw_{l_p}
\wedge d\overline{w_{k_1}}\wedge \cdots \wedge d\overline{w_{k_p}},
\]
where $C_p:=(\sqrt{-1})^p (-1)^{p(p-1)\over 2} p!$ and
\[ \Theta_{l_1,\ldots,l_p;\overline{k_1},\ldots,\overline{k_p}}(w)
:=\sum_{\begin{tiny}\begin{split}
&1\le i_1<\cdots<i_p\le n,\\
&1\le j_1<\cdots<j_p\le n\end{split}\end{tiny}}
\det\begin{pmatrix} (g_N)_{i_s\overline{j_t}}(F(w)) \end{pmatrix}_{1\le s,t\le p}\hspace{2cm}\]
\[\hspace{5cm}\cdot \det \left({\partial(F_{i_1},\ldots,F_{i_p})\over \partial(w_{l_1},\ldots,w_{l_p})}\right)\det \left(\overline{{\partial(F_{j_1},\ldots,F_{j_p})\over \partial(w_{k_1},\ldots,w_{k_p})}}\right).
\]
In terms of the local holomorphic coordinates, (\ref{Eq:General_pp}) becomes
\begin{equation}\label{Eq:System_Eqs_pp}
\Theta_{l_1,\ldots,l_p;\overline{k_1},\ldots,\overline{k_p}}(w)
= \lambda \det\begin{pmatrix} (g_M)_{l_s\overline{k_t}}(w) \end{pmatrix}_{1\le s,t\le p} 
\end{equation}
for $1\le l_1<\cdots<l_p \le m$ and $1\le k_1<\cdots<k_p\le m$.

By the same proof of Proposition 3.1 in Chan-Yuan \cite{CY25}, we have the following rigidity result on local holomorphic maps from $M$ to $N$ between Hermitian manifolds $M$ and $N$ preserving $(p,p)$-forms when $1\le p<\dim_{\mathbb C}(M)$.

\begin{proposition}[$\text{cf.\,Chan-Yuan \cite[Proposition 3.1]{CY25}}$]\label{thm:pp_form_to_Holo_iso}
Let $(M,\omega_M)$ and $($$N$, $\omega_N$$)$ be Hermitian manifolds of finite complex dimensions $m$ and $n$ respectively, where $\omega_M$ and $\omega_N$ denote the corresponding Hermitian forms.
Let $F:(M;x_0)\to (N;F(x_0))$ be a germ of holomorphic map such that $F^*\omega_N^p = \lambda \omega_M^p$ for some real constant $\lambda>0$ and some integer $p$, $1\le p\le \min\{m,n\}$. Then, $m\le n$.

If we assume in addition that $1\le p<m$, then $F^*\omega_N=\lambda^{1\over p} \omega_M$ so that $F:(M,\lambda^{1\over p} \omega_M;x_0)\to (N,\omega_N;F(x_0))$ is a local holomorphic isometry.
\end{proposition}

\section{Non-existence Theorems}

\subsection{The unit sphere bundles on Hermitian manifolds}
\label{Sec:Tech_USB}
Let $(M,\omega_M)$ and $(N,\omega_N)$ be Hermitian manifolds of complex dimensions $m$ and $n$ respectively.
Let $F:(M;x_0) \to (N;F(x_0))$ be a germ of holomorphic map such that $F^*\omega_N^p = \lambda \omega_M^p$ for some integer $p$, $1\le p\le \min\{m,n\}$, and some real constant $\lambda>0$.
Write $U$ for an open neighborhood of $x_0$ in the domain of $F$ in which we have the holomorphic coordinates $(w_1,\ldots,w_m)$ on $U$.
Define the unit sphere bundles
\begin{equation}\label{Eq:USphereB1}
S_1:=\left\{(w,v)\in \wedge^p T_{U}: \lambda \omega_M^p(w)(v,\overline v) = 1\right\},\; S_2:=\left\{(z,\xi)\in \wedge^p T_N: \omega_N^p(z)(\xi,\overline\xi) = 1\right\}.
\end{equation}
The map $F$ induces the holomorphic map $f:=(F,dF): \wedge^p T_{U}\to \wedge^p T_N$ defined by $f(w,v):= (F(w),dF_w(v))$ for all $(w,v)\in \wedge^p T_{U}$.
By (\ref{Eq:General_pp}), $f$ maps $S_1$ into $S_2$. Writing $\rho_1(w,v):=\lambda \omega_M^p(w)(v,\overline v) - 1$ and $\rho_2(z,\xi):=\omega_N^p(z)(\xi,\overline\xi) - 1$, we define
\[ U_1:=\left\{(w,v)\in \wedge^p T_{U}: \rho_1(w,v)< 0 \right\},\; 
V_1:=\left\{(w,v)\in \wedge^p T_{U}: \rho_1(w,v) > 0 \right\}
 \]
and
\[ U_2:=\left\{(z,\xi)\in \wedge^p T_N: \rho_2(z,\xi) < 0 \right\},\;
 V_2:=\left\{(z,\xi)\in \wedge^p T_N: \rho_2(z,\xi) > 0 \right\}.
\]
Then, by (\ref{Eq:General_pp}) and the definitions of $\rho_1$ and $\rho_2$, we have $\rho_2\circ f= \rho_1$ so that $f=(F,dF)$ maps $U_1$ into $U_2$, and maps $V_1$ into $V_2$.

As in Chan-Yuan \cite{CY25}, we make use of the method of CR geometry to study the CR map $f:S_1\to S_2$.
We have the following non-existence theorem by arguments similar to those in Chan-Yuan \cite[Proof of Theorem 3.2]{CY25}.

\begin{theorem}\label{thm:General_M_N_curva assump}
Let $(M,\omega_M)$ and $(N,\omega_N)$ be Hermitian manifolds of complex dimensions $m$ and $n$ respectively such that $(K_M^*,\det(g_M))$ is a positive Hermitian holomorphic line bundle, where $K_X^*$ denotes the anti-canonical line bundle of any complex manifold $X$ and $g_M$ denotes the Hermitian metric of $(M,\omega_M)$. Then, we have the following.
\begin{enumerate}
\item[(a)] If $(N,\omega_N)$ is of semi-negative holomorphic bisectional curvature, then there does not exist any germ of holomorphic map $F:(M;x_0)\to (N;F(x_0))$ such that $F^*\omega_N^p=\lambda \omega_M^p$ holds, where $\lambda>0$ is a real constant and $p$ is an integer with $1\le p\le \min\{m,n\}$.
\item[(b)] If $(N,\omega_N)$ is a period domain equipped with the canonical invariant metric, then there does not exist any germ of horizontal holomorphic map $F:(M;x_0)\to (N;F(x_0))$ such that $F^*\omega_N^p=\lambda \omega_M^p$ holds, where $\lambda>0$ is a real constant and $p$ is an integer with $1\le p\le \min\{m,n\}$.
\end{enumerate}
\end{theorem}
\begin{proof}
Assume the contrary that there exists a germ of holomorphic map $F:(M;x_0) \to (N;F(x_0))$ such that $F^*\omega_N^p = \lambda \omega_M^p$ holds.
Write $U\subset M$ for the domain of $F$.
Then, in the above settings we have the CR map $f:=(F,dF):S_1\to S_2$ between the unit sphere bundles $S_1$ and $S_2$ defined in (\ref{Eq:USphereB1}), and we write $\rho_1$ and $\rho_2$ for the defining functions of $S_1$ and $S_2$ respectively as in the above.
We have $\rho_2\circ f = \rho_1$, by $F^*\omega_N^p=\lambda \omega_M^p$.
Fix a point $(w_0,v_0)\in S_1$.
By Chan-Yuan \cite[Lemma 2.4]{CY25}, for any $\eta,\eta'\in T^{1,0}_{(w_0,v_0)}(S_1)$, we have
\begin{equation}\label{Eq:Levi form_General_M_N_curva assump}
 (\partial \overline\partial \rho_2)(f(w_0,v_0))(f_*(\eta),\overline{f_*(\eta')}) 
= (\partial \overline\partial\rho_1) (w_0,v_0)(\eta,\overline{\eta'}).
\end{equation}
We consider Case {\rm(a)}.
By computations similar to those in Chan-Yuan \cite[Section 2.3]{CY25} and Yuan \cite{Yu17}, since $(N,\omega_N)$ is of semi-negative holomorphic bisectional curvature, $S_2$ is weakly pseudoconvex.
By the identity $F^*\omega_N^p = \lambda \omega_M^p$, $F$ is of rank $\ge p$ at every point in $U$. 
Thus, we may choose a vector $\eta=(\eta_1,{\bf 0})\in T^{1,0}_{(w_0,v_0)}(S_1)$ with $\eta_1\in T_{w_0}(U)$ and $f_*(\eta)=(dF_{w_0}(\eta_1),\eta')\neq {\bf 0}$.
If $1\le p<m$, then by Proposition \ref{thm:pp_form_to_Holo_iso} $F$ would be a germ of holomorphic isometry, i.e., $F^*\omega_N=\lambda^{1\over p} \omega_M$, and thus $F^*\omega_N^m=\lambda^{m\over p} \omega_M^m$.
Therefore, we may assume without loss of generality that $p=m$.
Then, $F:U \to N$ is a holomorphic immersion.
By the assumption on the anti-canonical line bundle $K_M^*$, $(K_U^*,\det(g_M)|_{K_U^*})$ is a positive Hermitian \text{holomorphic} line bundle over $U$, and thus by computations similar to those in Chan-Yuan \cite{CY25} we have
\[ (\partial\overline\partial \rho_1) (w_0,v_0)(\eta,\overline{\eta})
= -C\Theta_{K_U^*}(\eta_1,\overline{\eta_1},v_0,\overline{v_0}) < 0 \]
for some real constant $C>0$.
On the other hand, $(\partial \overline\partial \rho_2)(f(w_0,v_0))(f_*(\eta),\overline{f_*(\eta)})\ge 0$ since $S_2$ is weakly pseudoconvex and $f_*(\eta)\neq {\bf 0}$.
This clearly contradicts with (\ref{Eq:Levi form_General_M_N_curva assump}), and the proof of the first assertion is complete.

Now, we consider Case {\rm(b)}, i.e., $(N,\omega_N)$ is a period domain equipped with the canonical invariant metric and $F$ is a horizontal holomorphic map, i.e., $F: U \to N$ is a holomorphic map such that $dF_x(T_x(M))$ lies inside the horizontal tangent space $T^h_{F(x)}(N)\subset T_{F(x)}(N)$ for all $x\in U$.
As in Case {\rm(a)}, we may assume without loss of generality that $p=m$ so that $F^*\omega_N^m=\lambda \omega_M^m$ holds, and thus $F$ is a holomorphic immersion.
Then, $F(U)\subset N$ is an integral submanifold for the horizontal distribution on the period domain $N$.
By a result of Carlson-Toledo \cite[Proposition 5.3]{CT1988}, we have $[dF_x(T_x(M)),dF_x(T_x(M))]=0$ for $x\in U$.
It follows from Peters \cite[Corollary 1.8]{Pet1990} that the holomorphic bisectional curvature $R_{v\bar v \xi \bar\xi}(N,\omega_N)\le 0$ along the directions $v,\xi\in dF_x(T_x(M))$, $x\in U$.
Hence, we still have $(\partial \overline\partial \rho_2)(f(w_0,v_0))(f_*(\eta),\overline{f_*(\eta)})$ $\ge$ $0$.
By the same argument as in the above, this contradicts with (\ref{Eq:Levi form_General_M_N_curva assump}), and the proof of the second assertion is complete.
\end{proof}

\begin{corollary}\label{Cor:Nonexist_mfd_+anti-can_to_BSD}
Let $(M,\omega_M)$ be a Hermitian manifold such that $(K_M^*,\det(g_M))$ is a positive Hermitian holomorphic line bundle.
Write $m:=\dim_{\mathbb C}(M)$.
Let $\Omega \Subset \mathbb C^n$ be a bounded symmetric domain with $\omega_\Omega$ being the K\"ahler form of the canonical K\"ahler metric $\sum_{j=1}^k \lambda_j {\rm Pr}_j^* g_{\Omega_j}$, where $\Omega_j$, $1\le j\le k$, are the irreducible factors of $\Omega$, ${\rm Pr}_j:\Omega \to \Omega_j$ is the canonical projection onto the $j$-th irreducible factor $\Omega_j$, and $\lambda_j>0$, $1\le j\le k$.
Then, there does not exist any germ of holomorphic map $F:(M;x_0)\to (\Omega; F(x_0))$ such that $F^*\omega_{\Omega}^p=\lambda \omega_M^p$ holds, where $\lambda>0$ is a real constant and $p$ is an integer with $1\le p\le \min\{m,n\}$.
\end{corollary}
\begin{proof}
It is clear that $(\Omega,\omega_\Omega)$ is of semi-negative holomorphic bisectional curvature, and thus the result follows from Theorem \ref{thm:General_M_N_curva assump}.
\end{proof}

\subsection{The Umehara algebra on bounded symmetric domains and its applications}

We have the following basic fact on Umehara algebra.
\begin{lemma}\label{lem:det_BSD_not in UmAlg}
Let $D\Subset \mathbb C^m$ be a bounded symmetric domain equipped with the canonical K\"ahler metric $g'_D:=\sum_{j=1}^k \lambda_j {\rm Pr}_j^* g_{D_j}$ for some $\lambda_j>0$, $1\le j\le k$, where $D_j$, $1\le j\le k$ are the irreducible factors of $D$, and ${\rm Pr}_j:D \to D_j$ is the canonical projection onto the $j$-th irreducible factor $D_j$, $1\le j\le k$.
Write $\omega_D$ for the K\"ahler form of $(D,g'_D)$ and $m_j:=\dim_{\mathbb C}(D_j)$.
The real-analytic function
\[ \phi_m(w):=\det\left( (g'_D)_{s\bar{t}}(w)\right)_{1\le s,t\le m} \]
on $D$ does not lie inside the Umehara algebra $\Lambda(D)$.
\end{lemma}
\begin{proof}
Assume the contrary that $\phi_m\in \Lambda(D)$.
We choose the Harish-Chandra coordinates $($$w_1$,$\ldots$, $w_m$$)$ $\in$ $\mathbb C^m$ on the bounded symmetric domain $D$ so that $D\cap \{(w_1,\ldots,w_m)\in \mathbb C^m: w_j=0\text{ for } j=2,\ldots,m\}$ is the minimal disk $\Delta_0$.
By restricting to $\Delta_0$ we obtain a function
$\widetilde\phi_m(\zeta):=\phi_m(\zeta,0,\ldots,0)$
for $\zeta\in \{w_1\in \mathbb C: (w_1,0\ldots,0)\in D\} = \{w_1\in \mathbb C:|w_1|^2<1\}=:\Delta$.
Then, we would have $\widetilde\phi_m\in \Lambda(\Delta)$ by the assumption.
We put $(z^1,\ldots,z^k)=(w_1,\ldots,w_m)$ with $z^j=(z^j_1,\ldots,z^j_{m_j})$ for $1\le j\le k$.
We have
\[ \omega_D
=-\sqrt{-1}\sum_{j=1}^k \lambda_j \partial\overline\partial \log \rho_j \]
where
\[ \rho_j = 1-\sum_{\mu=1}^{m_j} |z^j_\mu|^2 + \sum_{l=1}^{N'_j} (-1)^{\deg G^j_l} |G^j_l(z^j)|^2 \]
for some homogeneous polynomials $G^j_l$, $1\le l\le N'_j$, of degree $\ge 2$ (cf.\,Chan-Xiao-Yuan \cite[Section 4]{CXY17}).
By using Mok \cite[Lemma 3 and Lemma 5]{Mo14}, we compute
\[ \begin{tiny}
\widetilde\phi_m(\zeta)=
\det\begin{bmatrix}
{\lambda_1 \over (1-|\zeta |^2)^{2}} & & & & & \\
 &{\lambda_1\over 1-|\zeta|^2}{\bf I}_{p_1}  & & & & \\
 & & \lambda_1{\bf I}_{m_1-p_1-1} & &  &\\
 & & &  \lambda_2 {\bf I}_{m_2} & & \\
 & & & & \ddots & \\
 & & & & & \lambda_k {\bf I}_{m_k} 
\end{bmatrix}
={\lambda_1^{m_1}\cdots\lambda_k^{m_k}\over (1-|\zeta |^2)^{p_1+2}}, 
\end{tiny}\]
which does not lie in $\Lambda(\Delta)$ by Umehara \cite{Um88}, a plain contradiction.
Here, $p_1$ is a non-negative integer depending only on $D_1$.
Hence, $\phi_m\not\in \Lambda(D)$.
\end{proof}

By using Lemma \ref{lem:det_BSD_not in UmAlg}, we have the following non-existence result by using the Umehara algebra.

\begin{proposition}\label{pro:BSD_Proj_pp form}
Let $D\Subset \mathbb C^m$ be a bounded symmetric domain equipped with the canonical K\"ahler metric $g'_D:=\sum_{j=1}^k \lambda_j {\rm Pr}_j^* g_{D_j}$ for some $\lambda_j>0$, $1\le j\le k$, where $D_j$, $1\le j\le k$ are the irreducible factors of $D$, and ${\rm Pr}_j:D \to D_j$ is the canonical projection onto the $j$-th irreducible factor $D_j$, $1\le j\le k$.
Write $\omega_D$ for the K\"ahler form of $(D,g'_D)$ and $m_j:=\dim_{\mathbb C}(D_j)$.
Let $X$ be a projective manifold of complex dimension $n$ equipped with the K\"ahler metric $g_X$ such that there exists a holomorphic isometric embedding $(X,g_X)\hookrightarrow (\mathbb P^{N_1},\mu_1 g_{\mathbb P^{N_1}})\times \cdots \times (\mathbb P^{N_l},\mu_l g_{\mathbb P^{N_l}})$ for some real constants $\mu_j>0$, $1\le j\le l$, and some integers $N_j\ge 1$, $1\le j\le  l$, where $l$ is a positive integer.
Then, for any $p$, $1\le p \le \min\{m,n\}$, there does not exist a germ of holomorphic map $F:(D;x_0)\to (X;F(x_0))$ such that 
\begin{equation}\label{Eq:BSD_to_Proj_pp form}
F^*\omega_{X}^p  = \lambda \omega_{D}^p
\end{equation}
for some real constant $\lambda>0$.
\end{proposition}
\begin{proof}
Assume the contrary that there exists such a germ of holomorphic map $F$ such that (\ref{Eq:BSD_to_Proj_pp form}) holds.
The existence of such a map $F$ would imply the existence of a germ of holomorphic map $\hat F$ from $D$ to $Y:=\mathbb P^{N_1}\times \cdots \times \mathbb P^{N_l}$ such that $\hat F^*\omega_{Y}^p  = \lambda \omega_{D}^p$, where $\omega_{Y}$ denotes the K\"ahler form of $(\mathbb P^{N_1},\mu_1 g_{\mathbb P^{N_1}})\times \cdots \times (\mathbb P^{N_l},\mu_l g_{\mathbb P^{N_l}})$.
Thus, from now on we may assume $(X,g_X)=(\mathbb P^{N_1},\mu_1 g_{\mathbb P^{N_1}})\times \cdots \times (\mathbb P^{N_l},\mu_l g_{\mathbb P^{N_l}})$ and identify $X=Y=\mathbb P^{N_1}\times \cdots \times \mathbb P^{N_l}$.
By Proposition \ref{thm:pp_form_to_Holo_iso}, if $1\le p<m$, then $F$ would be a local holomorphic isometry from $(D,\lambda^{1\over p} g'_D)$ to $(X,g_X)=(\mathbb P^{N_1},\mu_1 g_{\mathbb P^{N_1}})\times \cdots \times (\mathbb P^{N_l},\mu_l g_{\mathbb P^{N_l}})$, i.e., $F^*\omega_X = \lambda^{1\over p} \omega_D$, and thus we may take the $m$-th exterior power to get $F^*\omega_X^m = \lambda^{m\over p} \omega_D^m$.
Therefore, we may assume without loss of generality that $p=m=\dim_{\mathbb C}(D)$.

Let $U\subset D$ be the domain of $F$.
Write $F=(F^1,\ldots,F^l)$ so that $F^j$ maps $U$ into $\mathbb P^{N_j}$ for $1\le j\le l$.
We may assume $x_0={\bf 0}$ and $F^j(x_0)=[1,0,\ldots,0]$ for $1\le j\le l$ without loss of generality by composing $F$ with holomorphic isometries of $(D,g'_D)$ and $(\mathbb P^{N_1},\mu_1 g_{\mathbb P^{N_1}})\times \cdots \times (\mathbb P^{N_l},\mu_l g_{\mathbb P^{N_l}})$.
We also write $N:=\sum_{j=1}^l N_j$ and $F=(F_1,\ldots,F_N)$ for convention.
Now, (\ref{Eq:BSD_to_Proj_pp form}) is equivalent to
\begin{equation}\label{Eq:D_to_X_in_product_PNs_preserve_pp1}
\begin{split}
&\sum_{I,J\in {\bf A}}\det( (g_{X})_{i_s\overline{j_t}}(F(w)) )_{1\le s,t\le p}
\det\left({\partial (F_{i_1},\ldots,F_{i_p})\over \partial(w_{l_1},\ldots,w_{l_p})}\right)
\det\left(\overline{{\partial (F_{j_1},\ldots,F_{j_p})\over \partial(w_{k_1},\ldots,w_{k_p})}}\right)\\
 =& \lambda \det\left( (g'_D)_{l_s\overline{k_t}}(w) \right)_{1\le s,t\le p} \end{split}
\end{equation}
for $1\le l_1<\cdots<l_p\le m$ and $1\le k_1<\cdots<k_p\le m$.
Note that
\[ \begin{pmatrix} (g_{X})_{i\overline{j}}(F(w)) \end{pmatrix}_{1\le i,j\le N}
= \begin{bmatrix}
M_1(w) & & \\
 & \ddots & \\
 & & M_l(w)
\end{bmatrix}, \]
where $M_j(w):=\mu_j \begin{pmatrix} (g_{\mathbb P^{N_j}})_{s\overline{t}}(F^j(w)) \end{pmatrix}_{1\le s,t\le N_j}$ for $1\le j\le l$.
By similar observation in Chan-Yuan \cite[Proof of Proposition 3.7]{CY25} and (\ref{Eq:D_to_X_in_product_PNs_preserve_pp1}), we have
\[ \left(\prod_{j=1}^l (1+\lVert F^j(w)\rVert^2)^{2p} \right)\cdot \det\left( (g'_D)_{l_s\overline{k_t}}(w)\right)_{1\le s,t\le p}\in \Lambda(U') \]
for $1\le l_1<\cdots<l_p\le m$ and $1\le k_1<\cdots<k_p\le m$, where $U'$ is a simply connected open neighborhood of ${\bf 0}$ in $U$ and $\lVert F^j(w)\rVert$ denotes the standard complex Euclidean norm of $F^j(w)$ in $\mathbb C^{N_j}$.
Let $U''$ be a simply connected open neighborhood of ${\bf 0}$ in $\{\zeta\in \mathbb C: (\zeta,{\bf 0})\in U'\}$.
Now, we put $p=m$.
By restricting to the minimal disk $\Delta_0:=D\cap \{(w_1,\ldots,w_m)\in \mathbb C^m: w_j=0\text{ for } j=2,\ldots,m\}$, we have $\phi_m(\zeta,{\bf 0})={\lambda_1^{m_1}\cdots \lambda_k^{m_k}\over (1-|\zeta|^2)^{p_1+2}}$ from the proof of Lemma \ref{lem:det_BSD_not in UmAlg}, and thus by Chan-Yuan \cite[Proof of Proposition 3.7]{CY25} we have
\[ \psi_m(\zeta):={\prod_{j=1}^l (1+\lVert F^j(\zeta,{\bf 0})\rVert^2)^{2m}\over (1-|\zeta|^2)^{p_1+2}} \in \Lambda(U''). \]
It remains to show that $\psi_m\not\in \Lambda(U'')$ in order to prove the non-existence of such a germ of holomorphic map $F$.
Actually, by the same arguments in Chan-Yuan \cite[Proof of Proposition 3.7,\;p.\,635]{CY25}, $\psi_m(\zeta)$ would be of infinite rank so that $\psi_m\not\in \Lambda(U'')$. This leads to a contradiction, and thus we conclude that there does not exist such a germ of holomorphic map $F$.
\end{proof}
\begin{remark}
One may ask if this non-existence theorem holds for a rational homogeneous manifold $X\cong G/P$ if we replace the restricted metric $g_X$ on $X \subset \mathbb P^N$ by the $G$-invariant canonical metric $h$ on $X$. The issue is that we do not have much information on the K\"ahler form of $(X,h)$ for checking things related to the Umehara algebra.
\end{remark}

\subsection{Local holomorphic maps between Hermitian symmetric manifolds of the compact type and bounded symmetric domains}
Let $X$ be a Hermitian symmetric manifold of the compact type.
Write $X=X_1\times\cdots \times X_l$, where each $(X_j,g_j)$, $1\le j\le l$, is an irreducible Hermitian symmetric manifold of the compact type whose maximum holomorphic sectional curvature equals $2$.
A K\"ahler metric $g_X$ on $X$ is said to be a \emph{canonical K\"ahler metric} if $g_X=\sum_{j=1}^l \mu_j {\rm Pr}_j^* g_j$ for some real constants $\mu_j>0$, $1\le j\le l$, where ${\rm Pr}_j:X\to X_j$ denotes the canonical projection onto the $j$-th irreducible factor $X_j$.
We fix such a canonical K\"ahler metric $g_X$ on $X$.
By the Nakagawa-Takagi theorem (cf.\,Mok \cite[Theorem 1, p.\,135]{Mo89}), there exists a holomorphic isometric embedding $\iota_j: (X_j,g_j) \hookrightarrow (\mathbb P^{N_j},g_{\mathbb P^{N_j}})$ for $1\le j\le l$.
Hence, we have a holomorphic isometric embedding $\iota:(X,g_X) \hookrightarrow (\mathbb P^{N_1},\mu_1 g_{\mathbb P^{N_1}})\times \cdots \times (\mathbb P^{N_l},\mu_l g_{\mathbb P^{N_l}})$ given by
\[ \iota(x_1,\ldots,x_l)
:= (\iota_1(x_1),\ldots,\iota_l(x_l)) \quad \forall\; x_j\in X_j, 1\le j\le l. \]
By Corollary \ref{Cor:Nonexist_mfd_+anti-can_to_BSD} and Proposition \ref{pro:BSD_Proj_pp form}, we have the following non-existence theorem of local holomorphic maps between Hermitian symmetric manifolds of the compact type and bounded symmetric domains preserving $(p,p)$-forms.

\begin{proposition}\label{Pro:non-exist_betw_HSSCT_BSD_ppforms}
Let $(X,\omega_X)$ be a Hermitian symmetric manifold of the compact type equipped with the canonical K\"ahler metric $g_X$.
Write $m:=\dim_{\mathbb C}(X)$.
Let $\Omega \Subset \mathbb C^n$ be a bounded symmetric domain and let $\omega_\Omega$ be the K\"ahler form of a canonical K\"ahler metric on $\Omega$.
If either {\rm(i)} $M=\Omega$ and $N=X$ or {\rm(ii)} $M=X$ and $N=\Omega$, then there does not exist any germ of holomorphic map $F:(M;x_0)\to (N; F(x_0))$ such that $F^*\omega_{N}^p=\lambda \omega_M^p$ holds, where $\lambda>0$ is a real constant and $p$ is an integer with $1\le p\le \min\{m,n\}$.
\end{proposition}
\begin{proof}
Since $(X,g_X)$ is of semi-positive holomorphic bisectional curvature and of positive holomorphic sectional curvature, $(K_X^*,\det(g_X))$ is a positive Hermitian holomorphic line bundle over $X$. In addition, there is a holomorphic isometric embedding $\iota:(X,g_X) \hookrightarrow (\mathbb P^{N_1},\mu_1 g_{\mathbb P^{N_1}})\times \cdots \times (\mathbb P^{N_l},\mu_l g_{\mathbb P^{N_l}})$ from the construction.
Hence, the result follows directly from Corollary \ref{Cor:Nonexist_mfd_+anti-can_to_BSD} and Proposition \ref{pro:BSD_Proj_pp form}.
\end{proof}

\begin{center}
\textsc{Acknowledgment}
\end{center}
We would like to thank Dr.\,Kwok-Kin Wong for helpful discussions, and Prof.\,Yuan Yuan for valuable comments.
We would also like to thank the referee for helpful comments.

\end{document}